\documentclass{amsart}
\usepackage{amssymb}

\newtheorem{theorem}{Theorem}[section]
\newtheorem{corollary}[theorem]{Corollary}
\newtheorem{lemma}[theorem]{Lemma}

\newtheorem{remark}[theorem]{Remark}

\newtheorem{definition}[theorem]{Definition}
\newtheorem{proposition}[theorem]{Proposition}

\def\bN {{\mathbb N}}

\def\sU {{\mathcal U}}
\def\sV {{\mathcal V}}
\def\sW {{\mathcal W}}
\def\sB {{\mathcal B}}
\def\sF {{\mathcal F}}

\DeclareMathOperator\cl{\operatorname{cl}}

\begin{document}
\title{Cardinal inequalities for $S(n)$-spaces}
\author{Ivan S. Gotchev}
\address{Central Connecticut State University\\
New Britain, CT 06050, USA}
\email{gotchevi@ccsu.edu}
\thanks{Some of the results in this paper were announced at the \emph{Spring 
Topology and Dynamical Systems Conference}, Berry College, Mount Berry, GA, March 17--19, 2005}
\begin{abstract}
Hajnal and Juh\'asz proved that if $X$ is a $T_1$-space, then $|X|\le 2^{s(X)\psi(X)}$, and if $X$ is a Hausdorff space, 
then $|X|\le 2^{c(X)\chi(X)}$ and $|X|\le 2^{2^{s(X)}}$. Schr\"oder sharpened the first two estimations by showing that 
if $X$ is a Hausdorff space, then $|X|\le 2^{Us(X)\psi_c(X)}$, and if $X$ is a Urysohn space, then 
$|X|\le 2^{Uc(X)\chi(X)}$. 

In this paper, for any positive integer $n$ and some topological spaces $X$, we define the cardinal functions 
$\chi_n(X)$, $\psi_n(X)$, $s_n(X)$, and $c_n(X)$ called respectively $S(n)$-character, $S(n)$-pseudocharacter, 
$S(n)$-spread, and $S(n)$-cellularity and using these new cardinal functions we show that the above-mentioned 
inequalities could be extended to the class of $S(n)$-spaces. We recall that the $S(1)$-spaces are exactly the Hausdorff 
spaces and the $S(2)$-spaces are exactly the 
Urysohn spaces.
\end{abstract}

\maketitle

\section{Introduction}

Hajnal and Juh\'asz, in 1969 (see \cite{HajJuh67}) proved that if $X$ is a $T_1$-space, then $|X|\le 2^{s(X)\psi(X)}$, and 
if $X$ is a Hausdorff space, then $|X|\le 2^{c(X)\chi(X)}$ and $|X|\le 2^{2^{s(X)}}$. Later Schr\"oder in \cite{Sch93} 
proved that if $X$ is a Hausdorff space, then the first inequality could be sharpened to $|X|\le 2^{Us(X)\psi_c(X)}$, and 
if $X$ is a Urysohn space, then the second estimation could be improved to $|X|\le 2^{Uc(X)\chi(X)}$. In 
\cite[Theorem 12]{Sch93} the author also claimed that if $X$ is a Urysohn space, then the third inequality could be 
improved to $|X|\le 2^{2^{Us(X)}}$. The proof of his Theorem 12 is based on \cite[Lemma 11]{Sch93} which states 
that if $X$ is a Urysohn space, then $\psi_c(X)\le 2^{Us(X)}$. Unfortunately, there is a gap in the proof of that Lemma 11 
and therefore the validity of the last two claims are unknown. We were able to prove both claims only when $X$ is an 
$S(3)$-space (see Corollaries \ref{CIG1} and \ref{CIG2}).

Let $n$ be a positive integer. In this paper, for some spaces $X$, we define the cardinal functions 
$\chi_n(X)$, $\psi_n(X)$, $s_n(X)$, and $c_n(X)$ called respectively $S(n)$-character, $S(n)$-pseudocharacter, 
$S(n)$-spread, and $S(n)$-cellularity and using them we extend the above-mentioned inequalities to the class of 
$S(n)$-spaces (see Definition \ref{D1}). In particular, we show that if $X$ is an $S(n)$-space, then 
$\left| X\right| \leq 2^{c_{n}(X)\chi_{n}(X)}$ (Theorems \ref{ThGHJ2} and \ref{ThGSc2}).
When $n=1$ we obtain Hajnal and Juh\'asz' inequality $|X|\le 2^{c(X)\chi(X)}$ for Hausdorff 
spaces $X$ and the case $n=2$ is the second part of Alas and Ko\v{c}inac' Theorem 2 in \cite{AlaKoc00}, which sharpens 
Schr\"oder's inequality $|X|\le 2^{Uc(X)\chi(X)}$ for Urysohn spaces $X$
(see Corollary \ref{CIGSch}).

Our Theorems \ref{THJIG} and \ref{TScIG} contain the generalizations to the class of $S(n)$-spaces of the corresponding 
inequalities of Hajnal and Juh\'asz and Schr\"oder mentioned above: Let $k$ be a positive integer. If $X$ is an 
$S(k-1)$-space, then $\left| X\right|\leq 2^{s_{2k-1}(X)\psi_{2k-1}(X)}$ and if $X$ is an $S(k)$-space, then 
$\left| X\right|\leq 2^{s_{2k}(X)\psi_{2k}(X)}$.

At the end of the paper we generalize Hajnal and Juh\'asz' inequality that if $X$ is a Hausdorff space, then 
$|X|\le 2^{2^{s(X)}}$. For that end we find upper bounds for the cardinality of the $S(n)$-pseudocharacter of a space, 
where $n$ is any positive integer. In Lemmas \ref{LHJIG} and \ref{LIG} we show that if $n$ is a positive integer and $X$ 
is an $S(n)$-space, then $\psi(X)\le 2^{s_n(X)}$ and, as a consequence, in Theorem \ref{TISG} we show that 
$\left| X\right|\leq 2^{s(X)\cdot 2^{s_{n}(X)}}$. In addition, we show that for every positive integer $k$, the following 
are true:

(a) If $X$ is an $S(3k)$-space, then $\psi_{2k}\le 2^{s_{2k}(X)}$ (Lemma \ref{LIG2G1}) and 
$\left| X\right|\leq 2^{2^{s_{2k}(X)}}$ (Theorem \ref{TIG2G1});

(b) If $X$ is an $S(3k-2)$-space, then $\psi_{2k-1}(X)\leq 2^{s_{2k-1}(X)}$ (Lemma \ref{LIG2G2}) and 
$\left| X\right|\leq 2^{2^{s_{2k-1}(X)}}$ (Theorem \ref{TIG2G2});

(c) If $X$ is an $S(3k-1)$-space, then $\psi_{2k-1}(X)\leq 2^{s_{2k}(X)}$ and $\psi_{2k}(X)\leq 2^{s_{2k-1}(X)}$ 
(Lemma \ref{LIG2G3} and Lemma \ref{LIG2G4}) and $\left| X\right|\leq 2^{s_{2k-1}(X)\cdot2^{s_{2k}(X)}}$, hence 
$\left| X\right|\leq 2^{2^{s_{2k-1}(X)}}$ (Theorem \ref{TIG2G3}).

More results about the cardinality of $S(n)$-spaces involving the cardinal functions $d_n(X)$, $t_n(X)$ and $bt_n(X)$, 
called respectively $S(n)$-density, $S(n)$-tightness and $S(n)$-bitightness, are contained in our paper \cite{GotKoc18}. 

\section{Preliminaries}

Notations and terminology in this paper are standard as in \cite{Eng89}, \cite{Juh80}, and \cite{Hod84}. Unless otherwise 
indicated, all spaces are assumed to be at least $T_1$ and infinite. $\alpha$, $\beta$, $\gamma$, and $\delta$ are ordinal 
numbers, while $\kappa$ denotes infinite cardinal; $\kappa^{+}$ is the successor cardinal of $\kappa$. As 
usual, cardinals are assumed to be initial ordinals, $\mathbb{N}^{+}$ denotes the set of all positive integers and 
$\mathbb{N}=\{0\}\cup \mathbb{N}^{+}$. If $X$ is a set, then $\mathfrak{P}(X)$ and $\left[ X\right] ^{\leq \kappa}$ 
denote the power set of $X$ and the collection of all subsets of $X$ having cardinality $\leq \kappa$, respectively.

\begin{definition}\label{D1}
Let $X$ be a topological space, $A\subset X$ and $n\in \mathbb{N}^{+}.$ A point $x\in X$ is \emph{$S(n)$-separated 
from} $A$ if there exist open sets $U_{i}$, $i=1,2,...,n$ such that $x\in U_{1}$, $\overline{U}_{i}\subset U_{i+1}$ for 
$i=1,2,...,n-1$ and $\overline{U}_{n}\cap A=\varnothing $; $x$ is $S(0)$-separated from $A$ if $x\notin \overline{A}.$ 
$X$ is an \emph{$S(n)$-space} \cite{Vig69} if every two distinct points in $X$ are $S(n)$-separated. 

Now, let $n\in \mathbb{N}$. The set $\cl_{\theta ^{n}}A=\{x\in X$ : $x$ is not $S(n)$-separated from $A\}$ is called  
$\theta ^{n}$\emph{-closure} of $A$ \cite{DikGiu88}. $A$ is $\theta ^{n}$\emph{-closed} \cite{PorVot73a} if 
$\cl_{\theta ^{n}}(A)=A$; $U\subset X$ is $\theta ^{n}$-\emph{open} if $X\backslash U$ is $\theta ^{n}$-closed; and 
$A$ is \emph{$\theta ^{n}$-dense} in $X$ if $\cl_{\theta ^{n}}(A)=X$.  
\end{definition}

It follows directly from Definition \ref{D1} that $S(1)$ is the class of Hausdorff spaces and $S(2)$ is the class of Urysohn 
spaces. Since we are going to consider here only $T_1$-spaces, $S(0)$-spaces will be exactly the $T_1$-spaces. Also, 
$\cl_{\theta ^{0}}(A)=\overline{A}$ and $\cl_{\theta ^{1}}(A)=\cl_{\theta }(A)$ - the so called 
\emph{$\theta $-closure of $A$} \cite{Vel66}. We want to emphasize here that in general, when 
$n \in \bN^+$, the $\theta^n$-closure operator is not idempotent, hence it is not a Kuratowski closure operator. In 
particular $\cl_\theta(\cl_\theta(A))\ne\cl_{\theta^2}(A)$. Finally, we 
note that if $X$ is a space and $n \in \bN^+$, then $\cl_{\theta^n}(U)=\cl_{\theta^{n-1}}(\overline{U})$ whenever $U$ is an open subset of $X$ (see 
\cite[Lemma 1.4(c)]{DikGiu88}).

In this paper it will be more convenient for us to think about $S(n)$-spaces in more 'symmetric' way similar to the way how 
$S(n)$-spaces are defined in \cite{DikTho95}, \cite{DikWat95} or \cite{PorVot73a} but here we are going to use different 
terminology and notation.

\begin{definition}\label{SnIG}
Let $X$ be a topological space, $U\subseteq X$, $x\in U$ and $k\in\bN^+$. 
We will say that $U$ is an \emph{$S(2k-1)$-neighborhood of $x$} if there exist open sets $U_{i}$, $i=1,2,...,k$, 
such that $x\in U_{1}$, $\overline{U}_{i}\subset U_{i+1}$, for $i=1,2,...,k-1$, and $U_{k}\subseteq U$. 
We will say that $U$ is an \emph{$S(2k)$-neighborhood of $x$} if there exist open sets $U_{i}$, $i=1,2,...,k$, 
such that $x\in U_{1}$, $\overline{U}_{i}\subset U_{i+1}$, for $i=1,2,...,k-1$, and $\overline{U}_{k}\subseteq U$. 

Let $n\in\bN^+$. When a set $U$ is an $S(n)$-neighborhood of a point $x$ and it is an open (closed) set in $X$, 
we will refer to it as open (closed) \emph{$S(n)$-neighborhood of $x$}. A set $U$ will be called $S(n)$-open 
($S(n)$-closed) if $U$ is open (closed) and there exists at least one point $x$ such that $U$ is an open (closed) 
$S(n)$-neighborhood of $x$. 
\end{definition}

\begin{remark}
We note that in what follows every $S(2k-1)$-open set $U$ in a space $X$, where $k\in\bN^+$, will be considered as a 
fixed chain of $k$ nonempty sets $U_{i}$, $i=1,2,...,k$, such that $\overline{U}_{i}\subset U_{i+1}$, for $i=1,2,...,k-1$, 
and $U_{k}\subseteq U$. (In fact, most of the time we will assume that $U_k=U$). Sometimes, when we need to refer to 
the first set $U_1$ in that chain we will use the notation $U(k)$, i.e. by definition $U(k)=U_1$.
\end{remark}

Now, using the terminology and notation introduced in Definition \ref{SnIG} it is easy to see that the following propositions 
are true.

\begin{proposition}\label{P1}
Let $X$ be a topological space, $x\in X$ and $k\in\bN^+$.

{\rm (a)} Every closed $S(2k-1)$-neighborhood of $x$ is a closed $S(2k)$-neighborhood of $x$.

{\rm (b)} Every $S(2k)$-neighborhood of $x$ contains a closed $S(2k)$-neighborhood of $x$; hence it contains a closed 
(and therefore an open) $S(2k-1)$-neighborhood of $x$. Thus, every $S(2k)$-neighborhood of $x$ is an 
$S(2k-1)$-neighborhood of $x$.

{\rm (c)} Every $S(2k+1)$-neighborhood of $x$ contains an open $S(2k+1)$-neighborhood of $x$; hence it contains an 
open (and therefore a closed) $S(2k)$-neighborhood of $x$.
Thus, every $S(2k+1)$-neighborhood of $x$ is an $S(2k)$-neighborhood of $x$.
\end{proposition}

\begin{proposition}\label{P2}
Let $X$ be a topological space and $k\in\bN^+$.

{\rm (a)} $X$ is an $S(2k-1)$-space if and only if every two distinct points of $X$ can be separated by disjoint (open) 
$S(2k-1)$-neighborhoods.

{\rm (b)} $X$ is an $S(2k)$-space if and only if every two distinct points of $X$ can be separated by disjoint closed 
$S(2k-1)$-neighborhoods.

{\rm (c)} $X$ is an $S(2k)$-space if and only if every two distinct points of $X$ can be separated by disjoint (closed) 
$S(2k)$-neighborhoods.

{\rm (d)} $X$ is an $S(2k+1)$-space if and only if every two distinct points of $X$ can be separated by disjoint open 
$S(2k)$-neighborhoods.
\end{proposition}

\begin{definition}\label{SnClIG}
Let $X$ be a topological space, $A\subseteq X$ and $k\in\bN^+$. We will say that a point $x$ is in the $S(2k-1)$-closure 
of $A$ if and only if every (open) $S(2k-1)$-neighborhood of $x$ intersects $A$ and we will say that a point $x$ is in the 
$S(2k)$-closure of $A$ if and only if every (closed) $S(2k)$-neighborhood (or equivalently, every closed
$S(2k-1)$-neighborhood) of $x$ intersects $A$. For $n\in\bN^+$, the 
$S(n)$-closure of $A$ will be denoted by $\theta_n(A)$. $A$ is $\theta_{n}$\emph{-closed} if 
$\theta_{n}(A)=A$ and $U\subset X$ is $\theta_{n}$-\emph{open} if $X\setminus U$ is $\theta_{n}$-closed, or 
equivalently, $U\subset X$ is $\theta_{n}$-\emph{open} if $U$ is an $S(n)$-neighborhood of every $x\in U$.
Finally, $A$ is $\theta_n$-dense in $X$ if $\theta_n(A)=X$.
\end{definition}

Clearly, for every  $n\in\bN^+$, every $\theta_n$-open set is open and every set of the form $\theta_n(A)$, where 
$A\subseteq X$, is a closed set. Also, it follows directly from Definition \ref{SnClIG} that 
$\theta_1(A)=\cl(A)=\overline{A}$ is the usual closure operator in $X$ 
and $\theta_2(A)=\cl_\theta(A)$ is the $\theta$-closure operator introduced by Veli\v{c}ko \cite{Vel66}. We also note 
that, except for the case $n=1$, for many $A\subset X$ we may have $\theta_n(\theta_n(A))\ne\theta_n(A)$, or in other 
words, the $\theta_n$-closure operator is not idempotent. More information about the 
closure operator $\theta_n$ is contained in \cite{DikGiu88}, \cite{DikTho95} and \cite{DikWat95}.

\begin{definition}\label{DPsi} Let $k\in \bN^+$ and $X$ be a topological space.

\emph{(a)} A family $\{U_\alpha :\alpha < \kappa\}$ of open $S(2k-1)$-neighborhoods of a point $x\in X$ will be called 
\emph{an open $S(2k-1)$-neighborhood base at the point $x$} if for every open $S(2k-1)$-neighborhood $U$ of $x$ 
there is $\alpha<\kappa$ such that $U_\alpha\subseteq U$. 

\emph{(b)} An $S(2k-1)$-space $X$ is of \emph{$S(2k-1)$-character} $\kappa$, denoted by 
$\chi_{2k-1}(X)$, if $\kappa$ is the smallest infinite cardinal such that for each point $x\in X$ there exists an open  
$S(2k-1)$-neighborhood base at $x$ with cardinality at most $\kappa$. In the case $k=1$ the $S(1)$-character 
$\chi_1(X)$ coincides with the usual character $\chi(X)$.

\emph{(c)} An $S(2k)$-space $X$ is of \emph{$S(2k)$-character} $\kappa$, denoted by 
$\chi_{2k}(X)$, if $\kappa$ is the smallest infinite cardinal such that for each point $x\in X$ there exists a family $\sV_x$
of closed $S(2k-1)$-neighborhoods of $x$ such that $|\sV_x|\le\kappa$ and if $W$ is an open $S(2k-1)$-neighborhood of 
$x$, then $\overline{W}$ contains a member of $\mathcal{V}_x$. In the case $k=1$ the $S(2)$-character 
$\chi_2(X)$ coincides with the cardinal function $k(X)$ defined in \cite{AlaKoc00}.

\emph{(d)} An $S(k-1)$-space $X$ is of \emph{$S(2k-1)$-pseudocharacter} $\kappa$, denoted by 
$\psi_{2k-1}(X)$, if $\kappa$ is the smallest infinite cardinal such that for each point $x\in X$ there exists a family 
$\{U_\alpha :\alpha < \kappa\}$ of $S(2k-1)$-open neighborhoods of $x$ such that 
$\{x\}=\bigcap \{U_\alpha :\alpha < \kappa \}.$ In the case $k=1$ the pseudocharacter $\psi_1(X)$ coincides with 
the usual pseudocharacter $\psi(X)$.

\emph{(e)} An $S(k)$-space $X$ is of \emph{$S(2k)$-pseudocharacter} $\kappa$, 
denoted by $\psi_{2k}(X)$, if $\kappa$ is the smallest infinite cardinal such that for each point $x\in X$ there exists a 
family $\{U_\alpha :\alpha < \kappa \}$ of $S(2k-1)$-open neighborhoods of $x$ such that 
$\{x\}=\bigcap \{\overline{U}_\alpha :\alpha < \kappa\}.$ In the case $k=1$ the pseudocharacter $\psi_2(X)$ 
coincides with the closed pseudocharacter $\psi_{c}(X)$.
\end{definition}

In relation to Definition \ref{DPsi}(c) we recall that for a topological space $X$, $k(X)$ is the smallest infinite cardinal
$\kappa$ such that for each point $x\in X$, there is a
collection $\mathcal{V}_x$ of closed neighborhoods of $x$ such that
$|\mathcal{V}_x|\le \kappa$ and if $W$ is a neighborhood of $x$, 
then $\overline{W}$ contains a member of $\mathcal{V}_x$ \cite{AlaKoc00}.
Clearly, $k(X)\le\chi(X)$.
As it was noted in \cite{AlaKoc00}, $k(X)$ is equal to the character of the semiregularization of $X$. We also note 
that if $k\in\bN^+$, then $\psi_{2k-1}(X)\le\psi_{2k}(X)\le\psi_{2k+1}(X)\le\psi_{2k+2}(X)$ and 
$\chi_{2k}(X)\le\chi_{2k-1}(X)$, whenever they are defined.

\begin{remark}\label{RemPsi}
Since the $\theta$-closure operator is not idempotent, for each positive integer $n$, a different type of pseudocharacter 
could be defined by requiring each point $x\in X$ to have a family of open neighborhoods 
$\{U_\alpha :\alpha < \kappa\}$ such that 
$\{x\}=\bigcap\{\cl_\theta(\cl_\theta(\ldots\cl_\theta(U_\alpha)\ldots)):\alpha < \kappa\},$ where the 
$\theta$-closure operator is repeated $n$ times. We used the notation 
$\psi_{\theta^n}(X)$ in \cite{GotUnp3} to denote that pseudocharacter and it is not difficult to see that 
$\psi_{\theta^n}(X)$ and the pseudocharacter $\psi_n(X)$ defined in Definition \ref{DPsi} are different.
\end{remark}

\begin{definition}\label{D6}
Let $k\in \bN^+$ and $X$ be a topological space.

\emph{(a)} We shall call a subset $D$ of $X$ \emph{$S(2k-1)$-discrete} if for every $x\in D$, there is an 
open $S(2k-1)$-neighborhood $U$ of $x$ such that $U\cap D=\{x\}$, and we define the \emph{$S(2k-1)$-spread} of 
$X$, denoted by $s_{2k-1}(X)$, to be $\sup \{\left| D\right| :D$ is $S(2k-1)$-discrete subset of $X\}+\aleph_0$. 

\emph{(b)} We shall call a subset $D$ of $X$ \emph{$S(2k)$-discrete} if for every $x\in D$, there is an 
open $S(2k-1)$-neighborhood $U$ of $x$ such that $\overline{U}\cap D=\{x\}$, and we define the 
\emph{$S(2k)$-spread} of $X$, denoted by $s_{2k}(X)$, to be $\sup \{\left| D\right| :D$ is $S(2k)$-discrete subset of 
$X\}+\aleph_0$. 
\end{definition}

It follows immediately from Definition \ref{D6} that a set $D$ in a topological space $X$ is discrete if and only if $D$ is 
$S(1)$-discrete and a set $D$ is Urysohn-discrete if and only if $D$ is $S(2)$-discrete. Therefore $s_1(X)$ is the usual 
spread $s(X)$ and $s_2(X)$ is the Urysohn spread $Us(X)$ defined in \cite{Sch93}.

\begin{definition}\label{DCelIG}
Let $X$ be a topological space and $k\in \bN^+$. 

\emph{(a)} We shall call a family $\sU$ of pairwise disjoint non-empty 
$S(2k-1)$-open subsets of $X$ \emph{$S(2k-1)$-cellular} and we define the \emph{$S(2k-1)$-cellularity} of  
$X$, denoted by $c_{2k-1}(X)$, to be $\sup \{\left| \mathcal{U}\right| :\mathcal{U}$ is an 
$S(2k-1)$-cellular family in $X\}+\aleph_0$.

\emph{(b)} We shall call a family $\sU$ of non-empty $S(2k-1)$-open subsets 
of $X$ \emph{$S(2k)$-cellular} if for every distinct $U_1,U_2\in\sU$ we have 
$\overline{U}_1\cap\overline{U}_2=\emptyset$ and we define the \emph{$S(2k)$-cellularity} of a space 
$X$, denoted by $c_{2k}(X)$, to be $\sup \{\left| \mathcal{U}\right| :\mathcal{U}$ is an  
$S(2k)$-cellular family in $X\}+\aleph_0$.
\end{definition}

\begin{remark} Let $X$ be a topological space.

\emph{(a)} $c_1(X)$ coincides with the usual cellularity $c(X)$ of the space $X$;

\emph{(b)} $c_2(X)$ coincides with the Urysohn cellularity $Uc(X)$ of $X$ introduced in \cite{Sch92};

\emph{(c)} If $n, m\in\bN^+$ and $n<m$, then $c_n(X)\ge c_m(X)$. For example of a space $X$ such that 
$c_1(X)>c_2(X)$ see \cite[Example 4]{Sch92}.

\emph{(d)} If in Definition \ref{DCelIG}(b) we require $\sU$ to be a family of pairwise disjoint non-empty 
$S(2k)$-open subsets of $X$ (similar to \ref{DCelIG}(a)), then, according to Proposition \ref{P1}, we would get the 
definitions of $S(2k+1)$-cellular family and $S(2k+1)$-cellularity defined in \ref{DCelIG}(a).
\end{remark}

\section{Cardinal inequalities for $S(n)$-spaces}

Hajnal and Juh\'asz proved that if $X$ is a $T_1$-space, then $|X|\le 2^{s(X)\psi(X)}$, and if $X$ is a Hausdorff space, 
then $|X|\le 2^{c(X)\chi(X)}$ and $|X|\le 2^{2^{s(X)}}$ (see \cite{HajJuh67}, \cite{Juh80} or \cite{Hod84}). 
Schr\"oder in \cite{Sch93} sharpened the first two estimations by 
showing that if $X$ is a Hausdorff space, then $|X|\le 2^{Us(X)\psi_c(X)}$, and if $X$ is a Urysohn space, then 
$|X|\le 2^{Uc(X)\chi(X)}$, or, in our notation, $\left|X\right| \leq 2^{s_2(X)\psi_2(X)}$ and 
$\left| X\right| \leq 2^{c_2(X)\chi (X)}$.

Below we formulate and prove the counterpart of the above inequalities for $S(n)$-spaces, where $n$ is 
any positive integer. 
We begin with extending the two inequalities $\left| X\right| \leq 2^{c(X)\chi (X)}$  and 
$\left| X\right| \leq 2^{c_2(X)\chi (X)}$ to $S(n)$-spaces. 

The following lemma for $k=1$ was proved by Charlesworth (see \cite[Theorem 3.2]{Cha77} or 
\cite[Proposition 3.4]{Hod84}). 

\begin{lemma}\label{LChIG}
Let $k\in \bN^+$, $X$ be a topological space and $\kappa=c_{2k-1}(X)$. If $\sU$ is a family of 
$S(2k-1)$-open subsets of $X$, then there exists a subfamily $\sV$ of $\sU$ such that $|\sV|\le\kappa$ and 
$\bigcup\{U(k):U\in \sU\}\subseteq \theta_{2k-1}(\bigcup\{V:V\in \sV\})$.
\end{lemma}

\begin{proof}
Let $\sW=\{W:W\ne\emptyset\text{ is $S(2k-1)$-open}, W\subseteq U, U\in \sU\}$. Using Zorn's lemma we can find 
a maximal $S(2k-1)$-cellular family $\sW'\subseteq\sW$. Then $|\sW'|\le\kappa$. For each $W\in\sW'$ we fix 
$U_W\in \sU$ such that $W\subseteq U_W$ and let $\sV=\{U_W:W\in\sW'\}$. Then $|\sV|\le\kappa$. Now, suppose that 
$\bigcup\{U(k):U\in \sU\}\subsetneq \theta_{2k-1}(\bigcup\{V:V\in \sV\})$. Then it is not difficult to see that there exist 
$U\in \sU$, $x\in U(k)\setminus \theta_{2k-1}(\bigcup\{V:V\in \sV\})$ and an 
$S(2k-1)$-open set $U_x$ such that $x\in U_x(k)\subset U(k)$, $U_x\subseteq U$ and 
$U_x\cap \bigcup\{V:V\in \sV\}=\emptyset$. Clearly, $U_x\in\sW$. Hence, $\{U_x\}\cup\sW'$ is an $S(2k-1)$-cellular 
family that properly contains $\sW'$ and therefore $\sW'$ is not maximal -- contradiction.
\end{proof}

The next lemma for $k=1$ was proved by Schr\"oder (see \cite[Lemma 7]{Sch92}). 

\begin{lemma}\label{LScIG}
Let $k\in \bN^+$, $X$ be a topological space and $\kappa=c_{2k}(X)$. If $\sU$ is a family of 
$S(2k-1)$-open subsets of $X$, then there exists a subfamily $\sV$ of $\sU$ such that $|\sV|\le\kappa$ and 
$\bigcup\{U(k):U\in\sU\}\subseteq\theta_{2k}(\bigcup\{\overline{V}:V\in \sV\})$.
\end{lemma}

\begin{proof}
Let $\sW=\{W:W\ne\emptyset\text{ is $S(2k-1)$-open}, W\subseteq U, U\in \sU\}$. Using Zorn's lemma we can find 
a maximal $S(2k)$-cellular family $\sW'\subseteq\sW$. Then $|\sW'|\le\kappa$. For each $W\in\sW'$ we fix 
$U_W\in \sU$ such that $W\subseteq U_W$ and let $\sV=\{U_W:W\in\sW'\}$. Then $|\sV|\le\kappa$. Now, suppose that 
$\bigcup\{U(k):U\in\sU\}\subsetneq \theta_{2k}(\bigcup\{\overline{V}:V\in \sV\})$. Then it is not difficult to see that 
there exist $U\in \sU$, $x\in U(k)\setminus \theta_{2k}(\bigcup\{\overline{V}:V\in \sV\})$ and an 
$S(2k-1)$-open set $U_x$ such that $x\in U_x(k)\subseteq U(k)$, $U_x\subseteq U$, and 
$\overline{U}_x\cap \bigcup\{\overline{V}:V\in \sV\}=\emptyset$. Clearly, $U_x\in\sW$. Hence,
$\{U_x\}\cup\sW'$ is an $S(2k)$-cellular 
family that properly contains $\sW'$ and therefore $\sW'$ is not maximal -- contradiction.
\end{proof}

The following two theorems, which proof is based on Lemma \ref{LChIG} and Lemma \ref{LScIG}, give an upper bound 
for the cardinality of an $S(n)$-space, where $n\in\bN^+$, as a function of the $S(n)$-cellularity and the character of the 
space. The inequality for Hausdorff spaces (the case $k=1$ in Theorem \ref{ThGHJ2}) was proved by Hajnal and Juh\'asz 
(see \cite{HajJuh67}, \cite[Theorem 2.15(b)]{Juh80} or \cite[Theorem 2]{Pol74}) and for Urysohn spaces (the case $k=1$ 
in Theorem \ref{ThGSc2}) was proved by Schr\"oder (see \cite[Theorem 9]{Sch93}). 

\begin{theorem}\label{ThGHJ2}
Let $k\in \bN^+$. If $X$ is an $S(2k-1)$-space, then $\left| X\right| \leq 2^{c_{2k-1}(X)\chi_{2k-1}(X)}$.
\end{theorem}
\begin{proof}
Let $k\in \bN^+$, $\kappa=c_{2k-1}(X)\chi_{2k-1}(X)$ and for every $x\in X$ let $\sB(x)$ be an open 
$S(2k-1)$-neighborhood base at $x$ such that $|\sB(x)|\le\kappa$.  
By transfinite recursion we construct an increasing sequence 
$(A_\alpha)_{\alpha<\kappa^+}$ of subsets of $X$ and a sequence of families of open sets 
$(\sU_\alpha)_{\alpha<\kappa^+}$ as follows:
\begin{itemize}
\item[(a)] $\sU_0=\emptyset$ and $A_0=\{x_0\}$, where $x_0$ is an arbitrary point of $X$;
\item[(b)] $|A_\alpha|\le\ 2^\kappa$ for every $\alpha<\kappa^+$;
\item[(c)] $\sU_\alpha=\bigcup\{\sB(x):x\in\bigcup_{\beta < \alpha}A_\beta\}$, for every $0<\alpha<\kappa^+$; 
\item[(e)] For $0<\alpha<\kappa^+$, if $\sW=\{\sV_\gamma:\gamma<\kappa\}$ is a family of subsets of $X$ such that 
each $\sV_\gamma$ is the union 
of at most $\kappa$ many elements of $\sU_\alpha$ and 
$\bigcup_{\gamma<\kappa}\theta_{2k-1}(\sV_\gamma)\ne X$, then we pick a point 
$x_\sW\in X\setminus\bigcup_{\gamma<\kappa}\theta_{2k-1}(\sV_\gamma)$. Let $E$ be the set of all such points 
$x_\sW$. Then we set $A_{\alpha}=E\cup\bigcup_{\beta<\alpha}A_\beta$. Since 
$|\bigcup_{\beta<\alpha}A_\beta|\le 2^\kappa$, we have $|\sU_\alpha|\le 2^\kappa$ and therefore 
$|E|\le ((2^\kappa)^\kappa)^\kappa=2^\kappa$. Thus $|A_\alpha|\le 2^\kappa$ and therefore (b) is satisfied.
\end{itemize}

Now, let $A=\bigcup_{\alpha<\kappa^+}A_\alpha$. Since $|A|\le\kappa^+\cdot 2^\kappa=2^\kappa$, to 
finish the proof it is sufficient to show that $X=A$. Suppose that 
there exists a point $y\in X\setminus A$ and let $\sB(y)=\{B_\delta:\delta<\kappa\}$. For each $\delta<\kappa$ let 
$\sV_\delta=\{V:V\in \sB(x), x\in A, B_\delta\cap V=\emptyset\}$. Since $X$ is an 
$S(2k-1)$-space and $\sB(x)$, for each $x\in X$, is an open $S(2k-1)$-neighborhood base at $x$, 
$A\subseteq\bigcup_{\delta<\kappa}\bigcup\{V(k):V\in\sV_\delta\}$. It follows from Lemma \ref{LChIG} that for every 
$\delta<\kappa$ we can find a subfamily $\sW_\delta$ of $\sV_\delta$ such that $|\sW_\delta|\le\kappa$ and 
$\bigcup\{V(k):V\in\sV_\delta\}\subseteq\theta_{2k-1}(\bigcup\sW_\delta)$. Using the fact that 
$|\bigcup_{\delta<\kappa}\sW_\delta|\le\kappa$ and that $\kappa^+$ is a regular cardinal we can find 
$\alpha<\kappa^+$ such that $\bigcup_{\delta<\kappa}\sW_\delta\subseteq \sU_\alpha$. Since for each $\delta<\kappa$
we have $B_\delta\cap\bigcup\{W:W\in\sW_\delta\}=\emptyset$, 
$y\notin\bigcup_{\delta<\kappa}\theta_{2k-1}(\bigcup\sW_\delta)$. Then it follows from the 
construction of the set $A_\alpha$ that there exists 
$x\in A_\alpha\cap \left(X\setminus \bigcup_{\delta<\kappa}\theta_{2k-1}(\bigcup\sW_\delta)\right)$.
Thus $x\in A_\alpha\subseteq A\subseteq\bigcup_{\delta<\kappa}\bigcup\{V(k):V\in\sV_\delta\}\subseteq\bigcup_{\delta<\kappa}\theta_{2k-1}(\bigcup\sW_\delta)$ -- contradiction.
\end{proof}

\begin{theorem}\label{ThGSc2}
Let $k\in \bN^+$. If $X$ is an $S(2k)$-space, then $\left| X\right| \leq 2^{c_{2k}(X)\chi_{2k}(X)}$.
\end{theorem}
\begin{proof}
Let $n\in \mathbb{N}$, $\kappa=c_{\theta ^n}(X)\chi_{2k}(X)$ and for every $x\in X$ let $\sB(x)$ be a family of open 
$S(2k-1)$-neighborhoods of $x$ witnessing the fact that $\chi_{2k}(X)\le\kappa$ (hence, $|\sB(x)|\le\kappa$).  
By transfinite recursion we construct an increasing sequence 
$(A_\alpha)_{\alpha<\kappa^+}$ of subsets of $X$ and a sequence of families of open sets 
$(\sU_\alpha)_{\alpha<\kappa^+}$ as follows:
\begin{itemize}
\item[(a)] $\sU_0=\emptyset$ and $A_0=\{x_0\}$, where $x_0$ is an arbitrary point of $X$;
\item[(b)] $|A_\alpha|\le\ 2^\kappa$ for every $\alpha<\kappa^+$;
\item[(c)] $\sU_\alpha=\bigcup\{\sB(x):x\in\bigcup_{\beta < \alpha}A_\beta\}$, for every $0<\alpha<\kappa^+$; 
\item[(e)] For $0<\alpha<\kappa^+$, if $\sW=\{\sV_\gamma:\gamma<\kappa\}$ is a family of subsets of $X$ such that 
each $\sV_\gamma$ is the union of the closures of at most $\kappa$ many elements of $\sU_\alpha$ and 
$\bigcup_{\gamma<\kappa}\theta_{2k}(\sV_\gamma)\ne X$, then we pick a point 
$x_\sW\in X\setminus\bigcup_{\gamma<\kappa}\theta_{2k}(\sV_\gamma)$. Let $E$ be the set of all such points 
$x_\sW$. Then we set $A_{\alpha}=E\cup\bigcup_{\beta<\alpha}A_\beta$. Since 
$|\bigcup_{\beta<\alpha}A_\beta|\le 2^\kappa$, 
we have $|\sU_\alpha|\le 2^\kappa$ and therefore $|E|\le ((2^\kappa)^\kappa)^\kappa=2^\kappa$. 
Thus $|A_\alpha|\le 2^\kappa$ and therefore (b) is satisfied.
\end{itemize}

Now, let $A=\bigcup_{\alpha<\kappa^+}A_\alpha$. Since $|A|\le\kappa^+\cdot 2^\kappa=2^\kappa$, to 
finish the proof it is sufficient to show that $X=A$. Suppose that 
there exists a point $y\in X\setminus A$ and let $\sB(y)=\{B_\delta:\delta<\kappa\}$. For each $\delta<\kappa$ let 
$\sV_\delta=\{V:V\in \sB(x), x\in A, \overline{B}_\delta\cap\overline{V}=\emptyset\}$. Since $X$ is an 
$S(2k)$-space and $\sB(x)$, for each $x\in X$, is a closed $S(2k-1)$-neighborhood base at $x$, 
$A\subseteq\bigcup_{\delta<\kappa}\bigcup\{V(k):V\in\sV_\delta\}$. It follows from Lemma \ref{LScIG} that for every 
$\delta<\kappa$ we can find a subfamily $\sW_\delta$ of $\sV_\delta$ such that $|\sW_\delta|\le\kappa$ and 
$\bigcup\{V(k):V\in\sV_\delta\}\subseteq\theta_{2k}\left(\bigcup\{\overline{W}:W\in\sW_\delta\}\right)$. Using the fact that 
$|\bigcup_{\delta<\kappa}\sW_\delta|\le\kappa$ and that $\kappa^+$ is a regular cardinal we can find 
$\alpha<\kappa^+$ such that $\bigcup_{\delta<\kappa}\sW_\delta\subseteq \sU_\alpha$. Since for each $\delta<\kappa$
we have $\overline{B}_\delta\cap\bigcup\{\overline{W}:W\in\sW_\delta\}=\emptyset$, 
$y\notin\bigcup_{\delta<\kappa}\theta_{2k}\left(\bigcup\{\overline{W}:W\in\sW_\delta\}\right)$. Then it follows from the 
construction of the set $A_\alpha$ that there exists 
$x\in A_\alpha\cap \left(X\setminus \bigcup_{\delta<\kappa}\theta_{2k}\left(\bigcup\{\overline{W}:W\in\sW_\delta\}\right)\right)$.
Thus $x\in A_\alpha\subseteq A\subseteq\bigcup_{\delta<\kappa}\bigcup\{V(k):V\in\sV_\delta\}\subseteq\bigcup_{\delta<\kappa}\theta_{2k}\left(\bigcup\{\overline{W}:W\in\sW_\delta\}\right)$ -- contradiction.
\end{proof}

The case $k=1$ in the above theorem improves Schr\"oder's inequality \cite[Theorem 9]{Sch92} and is the second part 
of Theorem 2 given in \cite{AlaKoc00} without a proof. We write it below with more familiar notation.

\begin{corollary}\label{CIGSch}
If $X$ is a Urysohn space, then $\left| X\right| \leq 2^{Uc(X)k(X)}$.
\end{corollary}

Now, we are going to formulate and prove the counterpart of the two inequalities $\left| X\right| \leq 2^{s(X)\psi (X)}$  
and $\left| X\right| \leq 2^{s_2(X)\psi_2 (X)}$ for $S(n)$-spaces. 
 
The following lemma for the case $k=1$ was proved by \v{S}apirovski\v{i} (see \cite[Proposition 4.8]{Hod84}). 

\begin{lemma}\label{LSapIG}
Let $k\in \mathbb{N}^{+}$, $X$ be a topological space, $\kappa=s_{2k-1}(X)$ and $C\subseteq X$. For each 
$x\in C$ let $U^x$ be an open $S(2k-1)$-neighborhood of $x$ and let $\mathcal{U}=\{U^x:x\in C\}$.
Then there exist an $S(2k-1)$-discrete subset $A$ of $C$ such that $\left| A\right| \leq \kappa$ and $C\subseteq\theta_{2k-1}(A)\cup \bigcup \left\{U^x:x\in A\right\}$.
\end{lemma}

\begin{proof}
Let $\{U_\alpha:\alpha<\beta\}$ be a well-ordering of the elements of $\sU$ and let $x_0\in C\cap U_0(k)$.
Suppose that the points $\{x_{\alpha_\gamma}:\gamma<\delta\}$ have already been chosen. Let 
$\alpha_\delta$ be the first ordinal greater than $\alpha_\gamma$, for each $\gamma<\delta$, such that 
there exists $x_{\alpha_\delta}\in (C\cap U_{\alpha_\delta}(k))\setminus\left[\left(\bigcup_{\gamma<\delta}U_{\alpha_\gamma}\right)\cup\theta_{2k-1}(\{x_{\alpha_\gamma}:\gamma<\delta\})\right]$. We finish this selection when such a 
point $x_{\alpha_\delta}$ does not exist, i.e. when $C\subseteq \left(\bigcup_{\gamma<\delta}U_{\alpha_\gamma}\right)\cup\theta_{2k-1}(\{x_{\alpha_\gamma}:\gamma<\delta\})$.
Therefore if $A$ is the set of selected points, then $C\subseteq\theta_{2k-1}(A)\cup \bigcup \left\{U^x:x\in A\right\}$.

To finish the proof we need to show that $A$ is an $S(2k-1)$-discrete set in $X$ and therefore $|A|\le \kappa$.
Let $x_{\alpha_{\gamma_0}}\in A$. Then 
$x_{\alpha_{\gamma_0}}\notin \theta_{2k-1}(\{x_{\alpha_\gamma}:\gamma<\gamma_0\})$. 
Hence, there exists an open $S(2k-1)$-neighborhood $V_{\alpha_{\gamma_0}}$ of $x_{\alpha_{\gamma_0}}$ such that 
$V_{\alpha_{\gamma_0}}\cap\{x_{\alpha_\gamma}:\gamma<\gamma_0\}=\emptyset$. Now, suppose that there exists
$\gamma_1$ such that $x_{\alpha_{\gamma_1}}\in A\cap V_{\alpha_{\gamma_0}}\cap U_{\alpha_{\gamma_0}}$. 
Since $x_{\alpha_{\gamma_1}}\in V_{\alpha_{\gamma_0}}$, we have $\gamma_1\ge \gamma_0$. From 
$x_{\alpha_{\gamma_1}}\in U_{\alpha_{\gamma_0}}$ and 
$x_{\alpha_{\gamma_1}}\in U_{\alpha_{\gamma_1}}(k)\setminus\left[\left(\bigcup_{\gamma<\gamma_1}U_{\alpha_\gamma}\right)\cup\theta_{2k-1}(\{x_{\alpha_\gamma}:\gamma<\gamma_1\})\right]$ 
we conclude that $U_{\alpha_{\gamma_0}}\subsetneq \bigcup_{\gamma<\gamma_1}U_{\alpha_\gamma}$, thus 
$\gamma_0\ge\gamma_1$. Hence $\gamma_1=\gamma_0$ and therefore 
$\{x_{\alpha_{\gamma_0}}\}= A\cap V_{\alpha_{\gamma_0}}\cap U_{\alpha_{\gamma_0}}$ i.e. $A$ is an 
$S(2k-1)$-discrete set in $X$.
\end{proof}

The next lemma for the case $n=1$ was proved by Schr\"oder (see \cite[Lemma 6]{Sch93}).

\begin{lemma}\label{LSchIG}
Let $k\in \mathbb{N}^{+}$, $X$ be a topological space, $\kappa=s_{2k}(X)$ and $C\subseteq X$. For each 
$x\in C$ let $U^x$ be an open $S(2k-1)$-neighborhood of $x$ and let $\mathcal{U}=\{U^x:x\in C\}$.
Then there exist an $S(2k)$-discrete subset $A$ of $C$ such that $\left| A\right| \leq \kappa $ 
and $C\subseteq\theta_{2k}(A)\cup \bigcup \left\{\overline{U}^x:x\in A\right\}$.
\end{lemma}
\begin{proof}
Let $\{U_\alpha:\alpha<\beta\}$ be a well-ordering of the elements of $\sU$ and let $x_0\in C\cap U_0(k)$.
Suppose that the points $\{x_{\alpha_\gamma}:\gamma<\delta\}$ have already been chosen. Let 
$\alpha_\delta$ be the first ordinal greater than $\alpha_\gamma$, for each $\gamma<\delta$, such that 
there exists $x_{\alpha_\delta}\in (C\cap U_{\alpha_\delta}(k))\setminus\left[\left(\bigcup_{\gamma<\delta}\overline{U}_{\alpha_\gamma}\right)\cup\theta_{2k}(\{x_{\alpha_\gamma}:\gamma<\delta\})\right]$. We finish this selection when such a 
point $x_{\alpha_\delta}$ does not exist, i.e. when $C\subseteq \left(\bigcup_{\gamma<\delta}\overline{U}_{\alpha_\gamma}\right)\cup\theta_{2k}(\{x_{\alpha_\gamma}:\gamma<\delta\})$.
Therefore if $A$ is the set of selected points, then 
$C\subseteq\theta_{2k}(A)\cup \bigcup \left\{\overline{U}^x:x\in A\right\}$.

To finish the proof we need to show that $A$ is an $S(2k)$-discrete set in $X$ and therefore $|A|\le \kappa$.
Let $x_{\alpha_{\gamma_0}}\in A$. Then $x_{\alpha_{\gamma_0}}\notin \theta_{2k}(\{x_{\alpha_\gamma}:\gamma<\gamma_0\})$. 
Hence, there exists an open $S(2k-1)$-neighborhood $V_{\alpha_{\gamma_0}}$ of $x_{\alpha_{\gamma_0}}$ such that 
$\overline{V}_{\alpha_{\gamma_0}}\cap\{x_{\alpha_\gamma}:\gamma<\gamma_0\}=\emptyset$. 
Now, suppose that there exists $\gamma_1$ such that 
$x_{\alpha_{\gamma_1}}\in A\cap \overline{V}_{\alpha_{\gamma_0}}\cap \overline{U}_{\alpha_{\gamma_0}}$. 
Since $x_{\alpha_{\gamma_1}}\in \overline{V}_{\alpha_{\gamma_0}}$, we have $\gamma_1\ge \gamma_0$. From 
$x_{\alpha_{\gamma_1}}\in \overline{U}_{\alpha_{\gamma_0}}$ and 
$x_{\alpha_{\gamma_1}}\in U_{\alpha_{\gamma_1}}(k)\setminus\left[\left(\bigcup_{\gamma<\gamma_1}\overline{U}_{\alpha_\gamma}\right)\cup\theta_{2k}(\{x_{\alpha_\gamma}:\gamma<\gamma_1\})\right]$ 
we conclude that 
$\overline{U}_{\alpha_{\gamma_0}}\subsetneq \bigcup_{\gamma<\gamma_1}\overline{U}_{\alpha_\gamma}$, 
thus $\gamma_0\ge\gamma_1$. Hence $\gamma_1=\gamma_0$ and therefore 
$\{x_{\alpha_{\gamma_0}}\}= A\cap \overline{V}_{\alpha_{\gamma_0}}\cap \overline{U}_{\alpha_{\gamma_0}}$ i.e. 
$A$ is an $S(2k)$-discrete set in $X$.
\end{proof}

The following theorem for the case $k=1$ was proved by Hajnal and Juh\'asz \cite{HajJuh67}.

\begin{theorem}\label{THJIG}
Let $k\in \mathbb{N}^{+}$. If $X$ is an $S(k-1)$-space, then 
$\left| X\right|\leq 2^{s_{2k-1}(X)\psi_{2k-1}(X)}$.
\end{theorem}
\begin{proof}
Let $k\in \mathbb{N}^+$, $\kappa=s_{2k-1}(X)\psi_{2k-1}(X)$ and for every $x\in X$ let 
$\sB(x)=\{B_\alpha^x:\alpha<\kappa\}$ be a family of open $S(2k-1)$-neighborhoods of $x$ such that  
$\{x\}=\bigcap\{B_\alpha^x:\alpha<\kappa\}$.  
By transfinite recursion we construct an increasing sequence 
$(A_\alpha)_{\alpha<\kappa^+}$ of subsets of $X$ and a sequence of families of $S(2k-1)$-open sets 
$(\sU_\alpha)_{\alpha<\kappa^+}$ as follows:
\begin{itemize}
\item[(a)] $\sU_0=\emptyset$ and $A_0=\{x_0\}$, where $x_0$ is an arbitrary point of $X$;
\item[(b)] $|A_\alpha|\le\ 2^\kappa$ for every $\alpha<\kappa^+$;
\item[(c)] $\sU_\alpha=\bigcup\{\sB(x):x\in\bigcup_{\beta < \alpha}A_\beta\}$, for every $0<\alpha<\kappa^+$; 
\item[(e)] For $0<\alpha<\kappa^+$, if $W$ is the union 
of at most $\kappa$ many elements of $\sU_\alpha$ and $\{F_\gamma:\gamma<\kappa\}$ is 
a family of subsets of $\bigcup_{\beta < \alpha}A_\beta$ such that $|F_\gamma|\le\kappa$ for 
every $\gamma<\kappa$, and $W\cup\bigcup_{\gamma<\kappa}\theta_{2k-1}(F_\gamma)\ne X$, then we pick a point 
$x_W\in X\setminus\left(W\cup\bigcup_{\gamma<\kappa}\theta_{2k-1}(F_\gamma)\right)$. Let $E$ be the set of all such 
points $x_W$. Then we set $A_{\alpha}=E\cup\bigcup_{\beta<\alpha}A_\beta$. Since 
$|\bigcup_{\beta<\alpha}A_\beta|\le 2^\kappa$, 
we have $|\sU_\alpha|\le 2^\kappa$ and therefore $|E|\le (2^\kappa)^\kappa=2^\kappa$. 
Thus $|A_\alpha|\le 2^\kappa$ and therefore (b) is satisfied.
\end{itemize}

Now, let $A=\bigcup_{\alpha<\kappa^+}A_\alpha$. Since $|A|\le\kappa^+\cdot 2^\kappa=2^\kappa$, to 
finish the proof it is sufficient to show that $X=A$. Suppose that 
there exists a point $y\in X\setminus A$ and let $\sB(y)=\{B_\delta^y:\delta<\kappa\}$. Then 
$X\setminus\{y\}=\bigcup\{C_\delta:\delta<\kappa\}$, where $C_\delta=X\setminus B_\delta^y$ 
whenever $\delta<\kappa$, and $y\notin\theta_{2k-1}(C_\delta)$ for each $\delta<\kappa$.
For each $\delta<\kappa$ let $D_\delta=A\cap C_\delta$ and for each $x\in D_\delta$ let $B^x\in\sB(x)$ be such that 
$y\notin B^x$. Now, we can apply 
Lemma \ref{LSapIG} to each $D_\delta$ and the family $\{B^x:x\in D_\delta\}$ to obtain a subset $G_\delta$ 
of $D_\delta$ such that $|G_\delta|\le\kappa$ and 
$D_\delta\subseteq \theta_{2k-1}(G_\delta)\cup\bigcup \left\{B^x:x\in G_\delta\right\}$.
Let $W=\bigcup_{\delta<\kappa}\bigcup\left\{B^x:x\in G_\delta\right\}$. Then 
$A\subseteq W\cup(\bigcup_{\delta<\kappa}\theta_{2k-1}(G_\delta))$ and 
$y\notin W\cup(\bigcup_{\delta<\kappa}\theta_{2k-1}(G_\delta))$.
Using the fact that $|\bigcup_{\delta<\kappa}G_\delta|\le\kappa$ and that $\kappa^+$ is a regular cardinal 
we can find $\alpha<\kappa^+$ such that $\bigcup_{\delta<\kappa}G_\delta\subseteq \sU_\alpha$. 
Then it follows from the construction of the set $A_\alpha$ that there exists 
$x\in A_\alpha\cap \left(X\setminus \left(W\cup\left(\bigcup_{\delta<\kappa}\theta_{2k-1}(G_\delta)\right)\right)\right)$.
Thus $x\in A_\alpha\subseteq A\subseteq W\cup(\bigcup_{\delta<\kappa}\theta_{2k-1}(G_\delta))$ -- contradiction.
\end{proof}

The next result for the case $k=1$ was proved by Schr\"oder (see \cite[Theorem 8]{Sch93}). 

\begin{theorem}\label{TScIG}
Let $k\in \mathbb{N}^{+}$. If $X$ is an $S(k)$-space, then 
$\left| X\right|\leq 2^{s_{2k}(X)\psi_{2k}(X)}$.
\end{theorem}
\begin{proof}
Let $k\in \mathbb{N}^+$, $\kappa=s_{2k}(X)\psi_{2k}(X)$ and for every $x\in X$ let 
$\sB(x)=\{B_\alpha^x:\alpha<\kappa\}$ be a family of open $S(2k-1)$-neighborhoods of $x$ such that  
$\{x\}=\bigcap\{\overline{B}_\alpha^x:\alpha<\kappa\}$.  
By transfinite recursion we construct an increasing sequence 
$(A_\alpha)_{\alpha<\kappa^+}$ of subsets of $X$ and a sequence of families of $S(2k-1)$-open sets 
$(\sU_\alpha)_{\alpha<\kappa^+}$ as follows:
\begin{itemize}
\item[(a)] $\sU_0=\emptyset$ and $A_0=\{x_0\}$, where $x_0$ is an arbitrary point of $X$;
\item[(b)] $|A_\alpha|\le\ 2^\kappa$ for every $\alpha<\kappa^+$;
\item[(c)] $\sU_\alpha=\bigcup\{\sB(x):x\in\bigcup_{\beta < \alpha}A_\beta\}$, for every $0<\alpha<\kappa^+$; 
\item[(e)] For $0<\alpha<\kappa^+$, if $W$ is the union 
of the closures of at most $\kappa$ many elements of $\sU_\alpha$ and $\{F_\gamma:\gamma<\kappa\}$ is 
a family of subsets of $\bigcup_{\beta < \alpha}A_\beta$ such that $|F_\gamma|\le\kappa$ for 
every $\gamma<\kappa$, and $W\cup\bigcup_{\gamma<\kappa}\theta_{2k}(F_\gamma)\ne X$, then we pick a point 
$x_W\in X\setminus\left(W\cup\bigcup_{\gamma<\kappa}\theta_{2k}(F_\gamma)\right)$. Let $E$ be the set of all such 
points $x_W$. Then we set $A_{\alpha}=E\cup\bigcup_{\beta<\alpha}A_\beta$. Since 
$|\bigcup_{\beta<\alpha}A_\beta|\le 2^\kappa$, 
we have $|\sU_\alpha|\le 2^\kappa$ and therefore $|E|\le (2^\kappa)^\kappa=2^\kappa$. 
Thus $|A_\alpha|\le 2^\kappa$ and therefore (b) is satisfied.
\end{itemize}

Now, let $A=\bigcup_{\alpha<\kappa^+}A_\alpha$. Since $|A|\le\kappa^+\cdot 2^\kappa=2^\kappa$, to 
finish the proof it is sufficient to show that $X=A$. Suppose that 
there exists a point $y\in X\setminus A$ and let $\sB(y)=\{B_\delta^y:\delta<\kappa\}$. Then 
$X\setminus\{y\}=\bigcup\{C_\delta:\delta<\kappa\}$, where $C_\delta=X\setminus \overline{B}_\delta^y$ 
whenever $\delta<\kappa$, and $y\notin\theta_{2k}(C_\delta)$ for each $\delta<\kappa$.
For each $\delta<\kappa$ let $D_\delta=A\cap C_\delta$ and for each $x\in D_\delta$ let $B^x\in\sB(x)$ be such that 
$y\notin \overline{B}^x$. Now, we can apply 
Lemma \ref{LSchIG} to each $D_\delta$ and the family $\{B^x:x\in D_\delta\}$ to obtain a subset $G_\delta$ 
of $D_\delta$ such that $|G_\delta|\le\kappa$ and 
$D_\delta\subseteq \theta_{2k}(G_\delta)\cup\bigcup \left\{\overline{B^x}:x\in G_\delta\right\}$.
Let $W=\bigcup_{\delta<\kappa}\bigcup\left\{\overline{B}^x:x\in G_\delta\right\}$. Then 
$A\subseteq W\cup(\bigcup_{\delta<\kappa}\theta_{2k}(G_\delta))$ and 
$y\notin W\cup(\bigcup_{\delta<\kappa}\theta_{2k}(G_\delta))$.
Using the fact that $|\bigcup_{\delta<\kappa}G_\delta|\le\kappa$ and that $\kappa^+$ is a regular cardinal 
we can find $\alpha<\kappa^+$ such that $\bigcup_{\delta<\kappa}G_\delta\subseteq \sU_\alpha$. 
Then it follows from the construction of the set $A_\alpha$ that there exists 
$x\in A_\alpha\cap \left(X\setminus \left(W\cup\left(\bigcup_{\delta<\kappa}\theta_{2k}(G_\delta)\right)\right)\right)$.
Thus $x\in A_\alpha\subseteq A\subseteq W\cup(\bigcup_{\delta<\kappa}\theta_{2k}(G_\delta))$ -- contradiction.
\end{proof}

Now we are going to formulate and prove the counterpart of the inequality $\left| X\right| \leq 2^{2^{s(X)}}$  
for $S(n)$-spaces. For that end we need to formulate and prove for $S(n)$-spaces the counterpart of the  
inequality $\psi(X)\leq 2^{s(X)}$, which is valid for every Hausdorff space $X$ \cite{Hod84}.

For the case $k=1$ of the following lemma see \cite[Proposition 4.11]{Hod84}.

\begin{lemma}\label{LHJIG}
Let $k\in \bN^{+}$. For every $S(2k-1)$-space $X$, $\psi(X)\leq 2^{s_{2k-1}(X)}$.
\end{lemma}
\begin{proof}
Let $\kappa=s_{2k-1}(X)$ and $x\in X$. For each $y\in X\setminus\{x\}$ let $U_x^y$ and $U_y$ be open 
$S(2k-1)$-neighborhoods of $x$ and $y$, respectively, such that $U_x^y\cap U_y=\emptyset$. 
Then for the set $C=X\setminus\{x\}$ and the family $\sU=\{U_y:y\in C\}$ we can use Lemma \ref{LSchIG} to find an 
$S(2k-1)$-discrete subset $A$ of $C$ such that $|A|\le\kappa$ and 
$C\subseteq \theta_{2k-1}(A)\cup\bigcup \left\{U_y:y\in A\right\}$. 

For each $z\in \theta_{2k-1}(A)$ and for each open $S(2k-1)$-neighborhood $U$ of $z$ we can choose a point 
$z_U\in U\cap U_z\cap A$ and let us 
denote by $F_z$ the resulting set. Then $F_z\subseteq A$, 
$z\in\theta_{2k-1}(F_z)\subseteq \theta_{2k-1}(\overline{U}_z)$ 
and since $U_x^z\cap \overline{U}_z=\emptyset$ we have $x\notin \theta_{2k-1}(\overline{U}_z)$. 
Thus $x\notin \theta_{2k-1}(F_z)$.

Let $\sF=\{F:F\subseteq A, x\notin\theta_{2k-1}(F)\}$ and for each $F\in\sF$ let $U_F^x=X\setminus\theta_{2k-1}(F)$. 
Then $|\sF|\le 2^\kappa$ and $\theta_{2k-1}(A)\subseteq\bigcup\{\theta_{2k-1}(F):F\in\sF\}$. Therefore 
$\{X\setminus\overline{U}_y:y\in A\}\cup\{U_F^x:F\in\sF\}$ is an open pseudobase of $x$ with cardinality 
$\le 2^\kappa$.
\end{proof}

\begin{lemma}\label{LIG}
Let $k\in \bN^{+}$. For every $S(2k)$-space $X$, $\psi(X)\leq 2^{s_{2k}(X)}$.
\end{lemma}
\begin{proof}
Let $\kappa=s_{2k}(X)$ and $x\in X$. For each $y\in X\setminus\{x\}$ let $U_x^y$ and $U_y$ be open 
$S(2k-1)$-neighborhoods of $x$ and $y$, respectively, such that $\overline{U}_x^y\cap\overline{U}_y=\emptyset$. 
Then for the set $C=X\setminus\{x\}$ and the family $\sU=\{U_y:y\in C\}$ we can use Lemma \ref{LSchIG} to find an 
$S(2k)$-discrete subset $A$ of $C$ such that $|A|\le\kappa$ and 
$C\subseteq \theta_{2k}(A)\cup\bigcup \left\{\overline{U}_y:y\in A\right\}$. 

For each $z\in \theta_{2k}(A)$ and for each open $S(2k-1)$-neighborhood $U$ of $z$ we can choose a point 
$z_U\in \overline{U}\cap \overline{U}_z\cap A$ and let us 
denote by $F_z$ the resulting set. Then $F_z\subseteq A$, $z\in\theta_{2k}(F_z)\subseteq \theta_{2k}(\overline{U}_z)$ 
and since $\overline{U}_x^z\cap \overline{U}_z=\emptyset$ we have $x\notin \theta_{2k}(\overline{U}_z)$. 
Thus $x\notin \theta_{2k}(F_z)$.

Let $\sF=\{F:F\subseteq A, x\notin\theta_{2k}(F)\}$ and for each $F\in\sF$ let $U_F^x=X\setminus\theta_{2k}(F)$. 
Then $|\sF|\le 2^\kappa$ and $\theta_{2k}(A)\subseteq\bigcup\{\theta_{2k}(F):F\in\sF\}$. Therefore 
$\{X\setminus\overline{U}_y:y\in A\}\cup\{U_F^x:F\in\sF\}$ is an open pseudobase of $x$ with cardinality 
$\le 2^\kappa$.
\end{proof}

As a corollary of the previous two lemmas we obtain the following theorem:

\begin{theorem}\label{TISG}
Let $n\in \mathbb{N}^{+}$. If $X$ is an $S(n)$-space, then $\left| X\right|\leq 2^{s(X)\cdot 2^{s_{n}(X)}}$.
\end{theorem}
\begin{proof}
The claim follows directly from the fact that every $S(n)$-space is a $T_1$-space, Hajnal and Juh\'asz theorem 
(Theorem \ref{THJIG} for $k=1$), Lemma \ref{LHJIG} for $n$-odd and Lemma \ref{LIG} for $n$-even.
\end{proof}

When $n=1$ in Theorem \ref{TISG} we obtain Hajnal and Juh\'asz theorem that if $X$ is a Hausdorff 
space, then $|X|\le 2^{2^{s(X)}}$ and when $n=2$, Theorem \ref{TISG} states that if $X$ is a Urysohn space, then 
$|X|\le 2^{s(X)\cdot 2^{Us(X)}}$.

We note that in \cite[Lemma 11]{Sch93} the author claims that if $X$ is a Urysohn space, then $\psi_c(X)\le 2^{Us(X)}$, 
or in our notation, $\psi_2(X)\le 2^{s_2(X)}$. Based on that lemma the author concluded that for every Urysohn space $X$ 
we have $|X|\le 2^{2^{Us(X)}}$. Unfortunately, there is a gap in the proof of Lemma 11 in \cite{Sch93} and therefore 
the question whether or not either one of both of these claims is true is open. What we can prove in that relation, in 
addition to the case $n=2$ of Theorem \ref{TISG}, is the case $k=1$ of Lemma \ref{LIG2G1} and Theorem \ref{TIG2G1}
(see Corollaries \ref{CIG1} and \ref{CIG2}).

\begin{lemma}\label{LIG2G1}
Let $k\in \bN^{+}$. For every $S(3k)$-space $X$, $\psi_{2k}(X)\leq 2^{s_{2k}(X)}$.
\end{lemma}
\begin{proof}
Let $\kappa=s_{2k}(X)$ and $x\in X$. Since $X$ is an $S(3k)$-space, for each $y\in X\setminus\{x\}$ let 
$U_1^y(x),\ldots,U_{k}^y(x)$, $U_1(y),\ldots,U_{2k}(y)$ be open sets in $X$ such that 
$x\in U_1^y(x)\subseteq\overline{U}_1^y(x)\subseteq\ldots\subseteq U_{k}^y(x)$, 
$y\in U_1(y)\subseteq\overline{U}_1(y)\subseteq\ldots\subseteq U_{2k}(y)$ and 
$\overline{U}_{k}^y(x)\cap \overline{U}_{2k}(y)=\emptyset$. Then for each $y\in C$, 
${U}_{k}(y)$ is an open $S(2k-1)$-neighborhood of $y$. 
Therefore for the set $C=X\setminus\{x\}$ and the family $\sU=\{U_{k}(y):y\in C\}$ we can use Lemma \ref{LSchIG} to 
find an $S(2k)$-discrete subset $A$ of $C$ such that $|A|\le\kappa$ and 
$C\subseteq \theta_{2k}(A)\cup\bigcup \left\{\overline{U}_{k}(y):y\in A\right\}$. 

For each $z\in \theta_{2k}(A)$ and for each open $S(2k-1)$-neighborhood $U$ of $z$ we can choose a point 
$z_U\in \overline{U}\cap \overline{U}_k(z)\cap A$ and let us 
denote by $F_z$ the resulting set. Then $F_z\subseteq A$, 
$z\in\theta_{2k}(F_z)\subseteq \theta_{2k}(\overline{U}_k(z))\subseteq\overline{U}_{2k}(z)$ 
and since $\overline{U}_{k}^z(x)\cap \overline{U}_{2k}(z)=\emptyset$ we have 
$x\notin \theta_{2k}(\overline{U}_{2k}(z))$, hence $x\notin \theta_{2k}(F_z)$ and clearly, 
$\overline{U}_{k}^z(x)\cap\theta_{2k}(F_z)=\emptyset$.

Let $\sF=\{F:F\subseteq A, x\notin\theta_{2k}(F)\}$ and for each $F\in\sF$ let $U_F^x$ be an open 
$S(2k-1)$-neighborhood of $x$ 
such that $\overline{U}_F^x\cap\theta_{2k}(F)=\emptyset$. 
Then $|\sF|\le 2^\kappa$ and $\theta_{2k}(A)\subseteq\bigcup\{\theta_{2k}(F):F\in\sF\}$. Therefore the existence of the 
family $\{\overline{U}_k^y(x):y\in A\}\cup\{\overline{U}_F^x:F\in\sF\}$ shows that the $S(2k)$-pseudocharacter at $x$ 
has cardinality $\le 2^\kappa$.
\end{proof}

\begin{corollary}\label{CIG1}
For every $S(3)$-space $X$, $\psi_2(X)\leq 2^{s_{2}(X)}$.
\end{corollary}

For the case $k=1$ of the following lemma see \cite[Proposition 4.11]{Hod84}.

\begin{lemma}\label{LIG2G2}
Let $k\in \bN^{+}$. For every $S(3k-2)$-space $X$, $\psi_{2k-1}(X)\leq 2^{s_{2k-1}(X)}$.
\end{lemma}
\begin{proof}
Let $\kappa=s_{2k-1}(X)$ and $x\in X$. Since $X$ is an $S(3k-2)$-space, for each $y\in X\setminus\{x\}$ let 
$U_1^y(x),\ldots,U_{k}^y(x)$, $U_1(y),\ldots,U_{2k-1}(y)$ be open sets in $X$ such that 
$x\in U_1^y(x)\subseteq\overline{U}_1^y(x)\subseteq\ldots\subseteq U_{k}^y(x)$, 
$y\in U_1(y)\subseteq\overline{U}_1(y)\subseteq\ldots\subseteq U_{2k-1}(y)$ and 
${U}_{k}^y(x)\cap \overline{U}_{2k-1}(y)=\emptyset$. Then for each $y\in C$, 
${U}_{k}(y)$ is an open $S(2k-1)$-neighborhood of $y$. 
Therefore for the set $C=X\setminus\{x\}$ and the family $\sU=\{U_{k}(y):y\in C\}$ we can use Lemma \ref{LSapIG} to 
find an $S(2k-1)$-discrete subset $A$ of $C$ such that $|A|\le\kappa$ and 
$C\subseteq \theta_{2k-1}(A)\cup\bigcup \left\{{U}_{k}(y):y\in A\right\}$. 

For each $z\in \theta_{2k-1}(A)$ and for each open $S(2k-1)$-neighborhood $U$ of $z$ we can choose a point 
$z_U\in {U}\cap {U}_k(z)\cap A$ and let us 
denote by $F_z$ the resulting set. Then $F_z\subseteq A$, 
$z\in\theta_{2k-1}(F_z)\subseteq \theta_{2k-1}(\overline{U}_k(z))\subseteq\overline{U}_{2k-1}(z)$ 
and since ${U}_{k}^z(x)\cap \overline{U}_{2k-1}(z)=\emptyset$ we have 
$x\notin \theta_{2k-1}(\overline{U}_{2k-1}(z))$, hence $x\notin \theta_{2k-1}(F_z)$ and clearly, 
${U}_{k}^z(x)\cap\theta_{2k-1}(F_z)=\emptyset$.

Let $\sF=\{F:F\subseteq A, x\notin\theta_{2k-1}(F)\}$ and for each $F\in\sF$ let $U_F^x$ be an open 
$S(2k-1)$-neighborhood of $x$ 
such that ${U}_F^x\cap\theta_{2k-1}(F)=\emptyset$. 
Then $|\sF|\le 2^\kappa$ and $\theta_{2k-1}(A)\subseteq\bigcup\{\theta_{2k-1}(F):F\in\sF\}$. Therefore the existence of 
the family $\{{U}_k^y(x):y\in A\}\cup\{{U}_F^x:F\in\sF\}$ shows that the $S(2k-1)$-pseudocharacter at 
$x$ has cardinality $\le 2^\kappa$.
\end{proof}

Similarly, for the case $3k-1$, $k\in \bN^{+}$, one can prove the following:

\begin{lemma}\label{LIG2G3}
Let $k\in \bN^{+}$. For every $S(3k-1)$-space $X$, $\psi_{2k-1}(X)\leq 2^{s_{2k}(X)}$.
\end{lemma}
\begin{proof}
Let $\kappa=s_{2k}(X)$ and $x\in X$. Since $X$ is an $S(3k-1)$-space, for each $y\in X\setminus\{x\}$ let 
$U_1^y(x),\ldots,U_{k}^y(x)$, $U_1(y),\ldots,U_{2k}(y)$ be open sets in $X$ such that 
$x\in U_1^y(x)\subseteq\overline{U}_1^y(x)\subseteq\ldots\subseteq U_{k}^y(x)$, 
$y\in U_1(y)\subseteq\overline{U}_1(y)\subseteq\ldots\subseteq U_{2k}(y)$ and 
${U}_{k}^y(x)\cap \overline{U}_{2k}(y)=\emptyset$. Then for each $y\in C$, 
${U}_{k}(y)$ is an open $S(2k-1)$-neighborhood of $y$. 
Therefore for the set $C=X\setminus\{x\}$ and the family $\sU=\{U_{k}(y):y\in C\}$ we can use Lemma \ref{LSchIG} to 
find an $S(2k)$-discrete subset $A$ of $C$ such that $|A|\le\kappa$ and 
$C\subseteq \theta_{2k}(A)\cup\bigcup \left\{\overline{U}_{k}(y):y\in A\right\}$. 

For each $z\in \theta_{2k}(A)$ and for each open $S(2k-1)$-neighborhood $U$ of $z$ we can choose a point 
$z_U\in \overline{U}\cap \overline{U}_k(z)\cap A$ and let us 
denote by $F_z$ the resulting set. Then $F_z\subseteq A$, 
$z\in\theta_{2k}(F_z)\subseteq \theta_{2k}(\overline{U}_k(z))\subseteq\overline{U}_{2k}(z)$ 
and since ${U}_{k}^z(x)\cap \overline{U}_{2k}(z)=\emptyset$ we have 
$x\notin \theta_{2k-1}(\overline{U}_{2k}(z))$, hence $x\notin \theta_{2k-1}(F_z)$ and clearly, 
${U}_{k}^z(x)\cap\theta_{2k}(F_z)=\emptyset$.

Let $\sF=\{F:F\subseteq A, x\notin\theta_{2k-1}(F)\}$ and for each $F\in\sF$ let $U_F^x$ be an open 
$S(2k-1)$-neighborhood of $x$ 
such that ${U}_F^x\cap\theta_{2k}(F)=\emptyset$. 
Then $|\sF|\le 2^\kappa$ and $\theta_{2k}(A)\subseteq\bigcup\{\theta_{2k}(F):F\in\sF\}$. Therefore the existence of 
the family $\{{U}_k^y(x):y\in A\}\cup\{{U}_F^x:F\in\sF\}$ shows that the $S(2k-1)$-pseudocharacter at 
$x$ has cardinality $\le 2^\kappa$.
\end{proof}

\begin{lemma}\label{LIG2G4}
Let $k\in \bN^{+}$. For every $S(3k-1)$-space $X$, $\psi_{2k}(X)\leq 2^{s_{2k-1}(X)}$.
\end{lemma}

\begin{proof}
Let $\kappa=s_{2k-1}(X)$ and $x\in X$. Since $X$ is an $S(3k-1)$-space, for each $y\in X\setminus\{x\}$ let 
$U_1^y(x),\ldots,U_{k}^y(x)$, $U_1(y),\ldots,U_{2k}(y)$ be open sets in $X$ such that 
$x\in U_1^y(x)\subseteq\overline{U}_1^y(x)\subseteq\ldots\subseteq U_{k}^y(x)$, 
$y\in U_1(y)\subseteq\overline{U}_1(y)\subseteq\ldots\subseteq U_{2k}(y)$ and 
$\overline{U}_{k}^y(x)\cap {U}_{2k}(y)=\emptyset$. Then for each $y\in C$, 
${U}_{k}(y)$ is an open $S(2k-1)$-neighborhood of $y$. 
Therefore for the set $C=X\setminus\{x\}$ and the family $\sU=\{U_{k}(y):y\in C\}$ we can use Lemma \ref{LSapIG} to 
find an $S(2k-1)$-discrete subset $A$ of $C$ such that $|A|\le\kappa$ and 
$C\subseteq \theta_{2k-1}(A)\cup\bigcup \left\{{U}_{k}(y):y\in A\right\}$. 

For each $z\in \theta_{2k-1}(A)$ and for each open $S(2k-1)$-neighborhood $U$ of $z$ we can choose a point 
$z_U\in {U}\cap {U}_k(z)\cap A$ and let us 
denote by $F_z$ the resulting set. Then $F_z\subseteq A$, 
$z\in\theta_{2k-1}(F_z)\subseteq \theta_{2k-1}({U}_k(z))\subseteq \overline{U}_{2k-1}(z)$ 
and since $\overline{U}_{k}^z(x)\cap {U}_{2k}(z)=\emptyset$ we have 
$x\notin \theta_{2k}({U}_{2k}(z))$, hence $x\notin \theta_{2k-1}(F_z)$ and clearly, 
$\overline{U}_{k}^z(x)\cap\theta_{2k-1}(F_z)=\emptyset$.

Let $\sF=\{F:F\subseteq A, x\notin\theta_{2k-1}(F)\}$ and for each $F\in\sF$ let $U_F^x$ be an open 
$S(2k-1)$-neighborhood of $x$ 
such that $\overline{U}_F^x\cap\theta_{2k-1}(F)=\emptyset$. 
Then $|\sF|\le 2^\kappa$ and $\theta_{2k-1}(A)\subseteq\bigcup\{\theta_{2k-1}(F):F\in\sF\}$. Therefore the existence of 
the family $\{\overline{U}_k^y(x):y\in A\}\cup\{\overline{U}_F^x:F\in\sF\}$ shows that the $S(2k)$-pseudocharacter at 
$x$ has cardinality $\le 2^\kappa$.
\end{proof}

Now, as corollaries of the previous four lemmas we obtain the following results:

\begin{theorem}\label{TIG2G1}
Let $k\in \bN^{+}$. If $X$ is an $S(3k)$-space, then $\left| X\right|\leq 2^{2^{s_{2k}(X)}}$.
\end{theorem}
\begin{proof}
The claim follows directly from the fact that every $S(3k)$-space is an $S(k)$-space, Theorem \ref{TScIG}, 
and Lemma \ref{LIG2G1}.
\end{proof}

\begin{corollary}\label{CIG2}
If $X$ is an $S(3)$-space, then $\left| X\right|\leq 2^{2^{s_{2}(X)}}$.
\end{corollary}

\begin{theorem}\label{TIG2G2}
Let $k\in \bN^{+}$. If $X$ is an $S(3k-2)$-space, then $\left| X\right|\leq 2^{2^{s_{2k-1}(X)}}$.
\end{theorem}
\begin{proof}
The claim follows directly from the fact that every $S(3k-2)$-space is an $S(k-1)$-space, Theorem \ref{THJIG}, 
and Lemma \ref{LIG2G2}.
\end{proof}

When $n=1$ in Theorem \ref{TIG2G2} we obtain again Hajnal and Juh\'asz theorem that if $X$ is a 
Hausdorff space, then $|X|\le 2^{2^{s(X)}}$.

\begin{theorem}\label{TIG2G3}
Let $k\in \bN^{+}$. If $X$ is an $S(3k-1)$-space, then $\left| X\right|\leq 2^{s_{2k-1}(X)\cdot2^{s_{2k}(X)}}$ and 
$\left| X\right|\leq 2^{2^{s_{2k-1}(X)}}$.
\end{theorem}
\begin{proof}
The claim follows directly from the fact that every $S(3k-1)$-space is an $S(k-1)$ and $S(k)$-space, Theorem \ref{THJIG}, 
Theorem \ref{TScIG}, Lemma \ref{LIG2G3}, Lemma \ref{LIG2G4} and the fact that $s_{2k}(X)\le s_{2k-1}(X)$ for 
every $k\in \bN^{+}$.
\end{proof}

\end{document}